\documentclass[12pt]{article}
\usepackage{amssymb}
\usepackage{amsmath}
\usepackage{lineno}
\usepackage{float}
\usepackage[mathscr]{eucal}
\usepackage{mathrsfs}
\usepackage{theorem}
\usepackage{mathtools}
\usepackage{comment}
\usepackage[shortlabels]{enumitem}
\usepackage{tikz-cd}
\usepackage{xparse}
\usepackage{geometry}
\NewDocumentCommand{\xrightarrows}{ O{}O{} }{%
\vcenter{\hbox{%
\begin{tikzpicture}
  \node[minimum width=1cm,minimum height=1ex,anchor=south,align=center] (a){\text{\vphantom{hg}#1}\\[0.5ex] \vphantom{hg}#2};
  \draw[<-] ([yshift=0.35ex]a.west) -- ([yshift=0.35ex]a.east);
  \draw[->] ([yshift=-0.35ex]a.west) -- ([yshift=-0.35ex]a.east);
\end{tikzpicture}
}}%
}%
\usepackage{xcolor}
\usepackage{graphicx}
\usepackage{tikz}
\usepackage{eso-pic}
\theoremstyle{change}
\newtheorem{theorem}{Theorem.}[section] 
\newtheorem{lem}[theorem]{Lemma.}
\newtheorem{prop}[theorem]{Proposition.}
\newtheorem{coro}[theorem]{Corollary.}
\theorembodyfont{\rmfamily}

\newtheorem{defi}[theorem]{Definition.}
\newtheorem{notation}[theorem]{Notation.}
\newtheorem{nothing}[theorem]{}

\newenvironment{proof}{\paragraph{Proof}}{\hfill$\square$}
\renewcommand{\le}{\leqslant} 

\newcommand{\Aut}{\mathrm{Aut}}

\newcommand{\calI}{\mathcal{I}}

\newcommand{\calP}{\mathcal{P}}

\newcommand{\calR}{\mathcal{R}}

\newcommand{\calS}{\mathcal{S}}

\newcommand{\Hom}{\mathrm{Hom}}

\newcommand{\id}{\mathrm{id}}
\newcommand{\im}{\mathrm{im}}

\newcommand{\Inn}{\mathrm{Inn}}

\newcommand{\Irr}{\mathrm{Irr}}

\newcommand{\Mat}{\mathrm{Mat}}

\newcommand{\Out}{\mathrm{Out}}

\newcommand{\Res}{\mathrm{Res}}

\newcommand{\Soc}{\mathrm{Soc}}

\title{Essential algebra of the shifted Burnside biset functor with abelian shift}
\author{
  Olcay CO\c{S}KUN\\
  \small Department of Mathematics\\
  \small Bo\u{g}azi\c{c}i University\\
  \small Bebek, 34342 Istanbul, T\"urkiye\\
  \small \texttt{olcay.coskun@bogazici.edu.tr}
  \and
  Ruslan MUSLUMOV\\
  \small SITE, ADA University\\
  \small 61 Ahmadbey Aghaoghlu Street\\
  \small Baku, AZ1008, Azerbaijan\\
  \small \texttt{rmuslumov@ada.edu.az}
}
\date{}

\begin{document}
\maketitle
\begin{abstract}
Let $T$ be a finite group and let $kB_T$ denote the shifted Burnside biset functor over a commutative ring $k$. We study the essential algebra $\widehat{kB_T}(G)$ of $kB_T$ at an arbitrary finite group $G$. We first establish criteria determining when a covering subgroup of $G\times G\times T$ factors, with respect to the star product, through a group of order strictly smaller than $|G|$; the first criterion places no hypothesis on $T$. As an application, we describe $\widehat{kB_T}(G)$ completely whenever no nontrivial quotient of $G$ is isomorphic to a section of $T$, extending the coprime case treated by Romero. When $T$ is abelian and $|T|$ is invertible in $k$, we prove that $\widehat{kB_T}(G)$ decomposes as a direct sum of matrix algebras over the group algebras of the groups $\Out_T(A)$, where $A$ runs over the linkage classes of reduced subgroups of $G\times T$, and we parametrize the simple modules of $\widehat{kB_T}(G)$.
\end{abstract}

\section{Introduction}
Let $T$ be a finite group and $k$ be a commutative ring with identity. The \emph{shifted Burnside functor} $kB_T$ is the biset functor assigning to a finite group $G$ the Burnside module $kB_T(G):= kB(G\times T)$. It is the Yoneda-Dress construction (or shift) of the Burnside functor $kB$ at the group $T$, introduced by Bouc in \cite{B2008}, and \cite[Section 8.2]{B} as an analogue, for biset functors, of Dress' construction for Mackey functors. The shift at $T$ preserves Green biset functors and their modules (see \cite[Lemma 4.4]{R2}). In particular $kB_T$ is a Green biset functor. In \cite{B2008} Bouc proved that the rationality of $p$-biset functors is preserved under shifting. 

The shifted Burnside functor is a natural object of study in its own right. In \cite{B3}, Bouc extended the theory of $B$-groups to the shifted setting. For a field $\mathbb F$ of characteristic zero, he described the lattice of ideals of the Green biset functor $\mathbb F B_T$ in terms of relative $B$-groups and parametrized the simple subquotients of $\mathbb FB_T$ by these groups.

A second source of motivation is in connection with the fibered biset functors. Recall from \cite{BC} that, for a finite abelian group $T$, a \emph{$T$-fibered $G$-set} is a $G\times T$-set whose $T$-action is free. Hence the $T$-fibered Burnside ring $B^T(G)$ is a subgroup of $B_T(G)$ spanned by the classes of $T$-free $G\times T$-sets. Moreover the composition product of fibered biset functors is just the truncation of the composition in $B_T$ truncated to the $T$-free part. From this point of view, it is natural to expect the theory of $kB_T$ to be parallel to the theory of $kB^T$ developed in \cite{BC}, when the shifting group $T$ is abelian. One of the aims of this paper is to make this expectation precise at the level of essential algebras. 

On the other hand, the essential algebra of a Green biset functor is a fundamental tool in the classification of its simple modules, see \cite{B}, \cite{BC}, \cite{CM} and \cite{R2}. It also appears naturally in the study of correspondence functors \cite{BT1} and diagonal $p$-permutation functors \cite{BY}. Recall that the essential algebra of $kB_T$ at $G$ is the quotient $\widehat{kB_T}(G)$ of the endomorphism algebra $kB_T(G\times G)$ by the ideal $\calI_T(G)$ of morphisms factoring through groups of order strictly smaller than $|G|$. In her paper \cite{R1}, Romero initiated the study of the essential algebras of the shifted Burnside biset functor. However the description obtained there is restricted to two special cases: when both $G$ and $T$ are abelian, or when $|G|$ and $|T|$ are relatively prime.

In the present paper, we describe the essential algebra $\widehat{kB_T}(G)$ for an arbitrary finite group $G$, assuming only that the shifting group $T$ is abelian ant that $|T|$ is invertible in $k$. Our main result, Theorem \ref{thm: matrixalg}, is a decomposition
\[
\widehat{kB_T}(G) \cong \bigoplus_{\widetilde{A}\in \calR_{G\times T}/\sim} \Mat_{|\widetilde{A}|}\bigl( k\Gamma_A \bigr)
\]
of the essential algebra as a direct sum of matrix algebras over some group algebras. Here $\calR_{G\times T}$ is the set of \emph{reduced} subgroups of $G\times T$, that is, the subgroups $A$ whose associated idempotent $\widehat{e}_A$ is nonzero in $\widehat{kB_T}(G)$. The sum runs over the \emph{linkage classes} $\widetilde{A}$ of reduced subgroups and $\Gamma_A$ is a finite group, which we prove to be isomorphic to the group $\Out_T(A)$ of outer automorphism of $A$ fixing the second component. As a consequence we obtain a parametrization of the simple $\widehat{kB_T}(G)$-modules by the pairs $(A, [V])$, where $A$ runs over the representatives of the linkage classes and $[V]$ runs over the isomorphism classes of simple $k\Gamma_A$-modules. 
See Theorem \ref{thm: simples}. We note that these results are of the same type as the corresponding results for fibered biset functors in \cite{BC} and for modules over the section Burnside functors in \cite{CM}.

A key result of Romero is a parametrization of certain subgroups $D$ of $G\times G\times T$ in terms of subgroups $A$ of $G\times T$ together with a surjective homomorphism from $A$ onto $G$. We refer to such subgroups $D$ as \emph{covering subgroups}, since they are closely related to the covering subgroups appearing in \cite{BC} and \cite{CM}. Although the construction of covering algebra analogous to  those of fibered biset functors is not available in the shifted setting, we determine precisely when a covering subgroup factors through a group of order strictly smaller than $|G|$ with respect to the star product. We call such subgroups $\ast$-factorable. More precisely we obtain two factorization criteria. The first one, Theorem \ref{thm: description2}, holds for an arbitrary shifting group $T$. Using this result, we derive a second criterion, Theorem \ref{thm: description3}, under a centrality condition that always holds when $T$ is abelian. We make it fully explicit in the abelian case in Corollary \ref{cor:abelian_nonvanish}. We prove that, when $T$ is abelian, a basis element of the endomorphism algebra vanishes in the essential algebra if and only if the corresponding covering subgroup is $\ast$-factorable. This gives an explicit basis of $\widehat{kB_T}(G)$ and from which we obtain the above main result by constructing a family of orthogonal central idempotents.

This paper is organized as follows. In Section 2, we fix the notation and recall the basic facts on Green biset functors and on the shifted Burnside functor that are used throughout the paper. In Section 3, we introduce covering subgroups and prove the factorization criteria; the first criterion is valid for an arbitrary shifting group $T$, and the remaining results of the section specialize it, first under a centrality condition and then to the abelian case. In Section~4, we characterize the vanishing of the basis elements of the essential algebra. We then extend, in Theorem~\ref{thm: nosection}, Romero's description of the essential algebra to the case where no nontrivial quotient of \(G\) is isomorphic to a section of \(T\). Throughout the remainder of the paper, beginning with Section~\ref{sec: abelian}, we assume that \(T\) is abelian and that \(|T|\in k^{\times}\). We introduce two families of idempotents, \(\widehat{e}_A\) and \(\widehat{f}_A\), indexed by reduced subgroups of \(G\times T\). In Section 5, we introduce the linkage relation and use it to define orthogonal central idempotents, which yield a decomposition of the essential algebra into a direct sum of two-sided ideals; we also define the groups $\Gamma_A$ and identify them with $\Out_T(A)$. Finally, in Section 6, we prove the decomposition theorem and the parametrization of the simple modules of the essential algebra.

\section{Setup and conventions}


Throughout, let $k$ be a commutative ring with identity. We recall basic facts on Green biset functors from \cite{B} and on the shifted Burnside functor from \cite{R1}. Following Romero, we write $GHT$ for the triple direct product $G\times H\times T$, the last factor always being the shifting group $T$. We keep the usual product notation when only two groups occur. 
\begin{notation}
Let $D \leq GHK$. For $i = 1,2,3$, we denote by $p_i(D)$ the projection of $D$ onto the $i$-th factor, namely $G$, $H$, and $K$, respectively. For any subset of indices, we denote by $p_{i,j}(D)$ the projection of $D$ onto the corresponding direct product. We also define
\[
k_1(D) := \{ g \in G \mid (g,1,1) \in D \},
\]
which is a normal subgroup of $p_1(D)$. Similarly, one defines $k_2(D)$ and $k_3(D)$.
Furthermore, we set
\[
k_{1,2}(D) := \{ (g,h) \in G \times H \mid (g,h,1) \in D \},
\]
which is a normal subgroup of $p_{1,2}(D)$. The subgroups $k_{1,3}(D)$ and $k_{2,3}(D)$ are defined in the same way.

When there are only two groups $G$ and $T$, for $A\le G\times T$ we have $p_1(A)\le G$, $p_2(A)\le T$, and
\[
k_1(A)=\{g\in G\mid (g,1)\in A\}\trianglelefteq p_1(A),\quad
k_2(A)=\{t\in T\mid (1,t)\in A\}\trianglelefteq p_2(A).
\]
\end{notation}
\begin{nothing}
\textbf{Star product.}
Let $G$, $H$, $K$, and $T$ be finite groups, and let $E \leq GHT$ and $F \leq HKT$. The \emph{star product} of $E$ and $F$, denoted by $E * F$, is defined by
\[
E * F = \{ (g,k,t) \in GKT \mid \exists h \in H : (g,h,t) \in E,\ (h,k,t) \in F \}\le GKT.
\]
Let $D\le GGT$. We say that $D$ is \emph{star-factorable or $\ast$-factorable} if there is a group $H$ with $|H| < |G|$ such that $D = E \ast F$ for some subgroups $E\le GHT$ and $F\le HGT$.   
\end{nothing}

\begin{nothing}{\textbf{Shifted Burnside functor.}}
Let $G$ be a finite group. Recall that the Burnside group $B(G)$ of $G$ is the Grothendieck group of the category of $G$-sets. It is the free abelian group on the isomorphism classes of transitive $G$-sets. We denote the image of a transitive $G$-set $G/H$ in $B(G)$ by $[G/H]$. We also write $kB(G) := k\otimes B(G)$ and $kB$ for the biset functor over $k$ of Burnside groups. It is also a Green biset functor and the category of biset functors is equivalent to the category of modules over $kB$, see \cite[Section 8.5 \& 8.6]{B}. 

Let $T$ be a fixed finite group. The \emph{shifted Burnside functor} $kB_T$ is the Yoneda--Dress construction at $T$ of the Burnside biset functor $kB$. In other words, by \cite[Section 8.2]{B}, the shifted Burnside functor is the biset functor that assigns to each group $G$ the Burnside module $kB(G \times T)$. For every $(H,G)$-biset $X$, the functor $kB_T$ defines a morphism of $k$-modules
\[
kB_T(X) : kB_T(G) \longrightarrow kB_T(H),
\]
given by mapping $y \in kB_T(G)$ to $kB(X \times T)(y)$. Moreover by \cite[Lemma 4.4]{R1} the functor $kB_T$ is a Green biset functor. 

As described in \cite[Proposition 8.6.1]{B} for modules over a Green biset functor, the category of $kB_T$-modules is equivalent to the category of $k$-linear functors from the category $\calP_{kB_T}$ to the category of $k$-modules. The objects of $\calP_{kB_T}$ are finite groups and morphisms from $G$ to $H$ are given by $kB_T(H\times G)$. The composition of morphisms $\alpha \in kB_T(H \times G)$ and $\beta \in kB_T(K \times H)$ in this category is
$$\beta \circ \alpha = kB_T\bigl(K \times \overset{\leftarrow}{H} \times G\bigr)(\beta \times \alpha),$$
where $\overset{\leftarrow}{H}$ is the $(1, H \times H)$-biset whose underlying set is $H$, with right action of $H \times H$ defined by $h \cdot (x,y) = y^{-1} h x \quad \text{for } h, x,y \in H$. If $X$ is a $GHT$-set and $Y$ is an $HKT$-set, we write $X \times_H Y$ for the usual amalgamated product over $H$, that is, the set of $H$-orbits on $X \times Y$ under the action $(x,y)\cdot h = (x\cdot h, h^{-1}\cdot y)$.
Then we define $X \times_H^d Y$ to be the $G \times K$-set $X \times_H Y$ equipped with the diagonal action of $T$ given by $t \cdot [x,y] = [tx, ty], \quad \text{for all } t \in T$. With this notation, the composition $\circ$ above is the $k$-linear extension of $\times^d$. The following lemma, which is \cite[Lemma 4.5]{R2}, describes the composition of two transitive basis elements as the usual Mackey product formula.
\end{nothing}

\begin{lem}[Romero] \label{Mackey}
Let $G, H, K,$ and $T$ be finite groups. Let $E\le GHT$ and $F\le HKT$. Then there is an isomorphism
\[
\Big(\frac{GHT}{E}\Big) \times_H^d \Big(\frac{HKT}{F}\Big) \;\cong\; \bigsqcup_{(h,t)} \Big(\frac{GKT}{E*^{(h, 1, t)}F}\Big)
\]
of $GKT$-sets, where $(h,t)$ runs through a set of representatives of the double cosets
$p_{2,3}(E) \backslash (H \times T) / p_{1,3}(F)$.
\end{lem}

\begin{nothing}
The main object of this paper is the endomorphism algebra of a finite group $G$ in the category $\calP_{kB_T}$. It is the $k$-algebra $kB_T(G \times G)$ with the multiplication induced by the composition of morphisms via the product $\times_G^d$, and with the identity element $[GGT/(\Delta(G) \times T)]$.    
\end{nothing}


\section{Factorization Criteria}


Let $G$ and $T$ be fixed finite groups. Recall that the essential algebra is the quotient of the endomorphism algebra $kB_T(G\times G)$ by the ideal $\calI_T(G)$ of morphisms factoring through groups of order strictly smaller that $|G|$. By \cite[Section 3]{R1}, a basis element $[GGT/D]$ factors through such a group whenever $p_1(D)\neq G$, $p_2(D)\neq G$, $k_1(D)\neq 1$, or $k_2(D)\neq 1$. We therefore call a subgroup $D\le GGT$ a \emph{covering subgroup} when $p_1(D) = p_2(D) = G$ and $k_1(D) = k_2(D) = 1$ in analogy with the covering elements of \cite{BC} and \cite{CM}. These are precisely the subgroups isolated by Romero, and the only ones whose classes can be nonzero in the essential algebra. 

The covering conditions are necessary but not sufficient, a covering subgroup may be $\ast$-factorable. Moreover, in contrast to \cite{BC} or \cite{CM}, the covering subgroups do not generate a covering algebra, so their factorization must be analyzed directly. We first establish a criterion valid for an arbitrary shifting group $T$. We then formulate it under a centrality condition that holds automatically when $T$ is abelian. Finally make it fully explicit in the abelian case.

\begin{nothing}\label{covering}
Let $D \leq GGT$ be a {covering subgroup}, that is,
\begin{equation*}
p_1(D) = p_2(D) = G \quad \text{and} \quad k_1(D) = k_2(D) = 1.
\end{equation*}
As in \cite[Section 3]{R1}, it is straightforward to verify that these conditions are equivalent to $D$ having the form
\begin{equation*}
D = \{ (\alpha(g,t), g, t) \mid (g,t) \in A \},
\end{equation*}
where $A = p_{2,3}(D) \leq G \times T$ with $p_1(A) = G$, and
$\alpha : A \twoheadrightarrow G$
is a surjective group homomorphism such that
\[
\ker(\alpha) \cap \bigl(k_1(A) \times 1\bigr) = 1.
\]
With this notation, to each covering subgroup $D$ of $GGT$, we associate the unique pair $(A,\alpha)$.

On the other hand we recall that, by Goursat's Lemma, any subgroup $A \leq G \times T$ corresponds to a quintuple $\bigl(p_1(A),\, k_1(A),\, \varphi,\, k_2(A),\, p_2(A)\bigr)$,
where $\varphi : p_1(A)/k_1(A) \to p_2(A)/k_2(A)$ is an isomorphism.  
\end{nothing}

\begin{nothing}\label{section3 conventions}
Throughout this section $D$ denotes a covering subgroup of $GGT$ with associated pair $(A,\alpha)$
and Goursat quintuple $(G,N,\varphi,K,T_0)$; thus $N=k_1(A)$, $K=k_2(A)$, $T_0=p_2(A)$, and
$\varphi\colon G/N \xrightarrow{\ \sim\ } T_0/K$. By Goursat's lemma, $|A| = |G|\,|K| = |N|\,|T_0|$.
\end{nothing}
\begin{nothing}\textbf{Invariants.}\label{sec: invariants}
Let \(G\) be a finite group and let $D$ be a covering subgroup of $GGT$. Associated to \(D\), we define its \emph{left} and \emph{right} invariants
\[
L := L_D := p_{1,3}(D) \leq GT \quad \text{and} \quad R := R_D := p_{2,3}(D) \leq GT.
\]
We also set the opposite subgroup to be 
\[
D^{\mathrm{op}} := \{(h,g,t) \in GGT \mid (g,h,t) \in D\}\le GGT.
\]
\end{nothing}
\begin{lem}\label{lem: productofcovering}
Let \(E,F\leq GGT\) be covering subgroups. Let \((A,\alpha)\) and \((B,\beta)\) be the corresponding pairs associated to \(E\) and \(F\), respectively. Then
\begin{equation*}
E*F
=
\{(\alpha(\beta(g,t),t),g,t)\mid (g,t)\in B \text{ and } (\beta(g,t),t)\in A\}.
\end{equation*}
\end{lem}

\begin{proof}
Let \((g,t)\in B\) such that \((\beta(g,t),t)\in A\). Then
\((\alpha(\beta(g,t),t),\beta(g,t),t)\in E\)
and
\((\beta(g,t),g,t)\in F\).
Hence
\((\alpha(\beta(g,t),t),g,t)\in E*F\).

Conversely, let \((g_1,g_2,t)\in E*F\). By definition of the star product, there exists \(h\in G\) such that
\((g_1,h,t)\in E\)
and
\((h,g_2,t)\in F\).
Since \(E\) and \(F\) correspond to the pairs \((A,\alpha)\) and \((B,\beta)\), respectively, we obtain
\(g_1=\alpha(h,t)\)
and
\(h=\beta(g_2,t)\).
Therefore
\[
(g_1,g_2,t)
=
(\alpha(\beta(g_2,t),t),g_2,t),
\]
showing that every element of \(E*F\) is of the required form.
\end{proof}

The subgroups $A\le G\times T$ with $p_1(A)=G$ will play a central role from Section~4 onward. The following lemma determines when such a subgroup is normal in the ambient direct product.

\begin{lem}\label{lem: Anormal}
Let $H$ be a finite group and let $A\le H\times T$ with $p_1(A)=H$. The following are equivalent:
\begin{enumerate}[leftmargin=2em]
\item $A\unlhd H\times T$;
\item $1\times T\le N_{H\times T}(A)$;
\item $[p_2(A),\,T]\le k_2(A)$.
\end{enumerate}
In particular, if $T$ is abelian, then every subgroup $A\le H\times T$ with $p_1(A)=H$ is normal in $H\times T$.
\end{lem}
\begin{proof}
Since $1\times T$ is normal in $H\times T$, the product $A\cdot(1\times T)$ is a subgroup, and it equals $p_1(A)\times T=H\times T$. As $A$ normalizes itself, $N_{H\times T}(A)\supseteq A$, so $N_{H\times T}(A)=H\times T$ if and only if $1\times T\le N_{H\times T}(A)$; this proves the equivalence of (1) and (2). For $(a,s)\in A$ and $t\in T$ we have ${}^{(1,t)}(a,s)=(a,tst^{-1})$, and since $(a,s)\in A$, the element $(a,tst^{-1})$ lies in $A$ if and only if $(a,s)^{-1}(a,tst^{-1})=(1,\,s^{-1}tst^{-1})$ does, if and only if $s^{-1}tst^{-1}\in k_2(A)$. As every $s\in p_2(A)$ occurs in some $(a,s)\in A$, this proves the equivalence of (2) and (3). The last statement is immediate, since $[p_2(A),T]=1$ when $T$ is abelian.
\end{proof}

\begin{theorem}
\label{thm: description2}
Assume conventions from Section \ref{section3 conventions}. Then $D$ is $\ast$-factorable if and only if there is a normal subgroup $\widetilde M \trianglelefteq A$ with
\[
\widetilde M \cap (N\times 1) = 1 \qquad\text{and}\qquad |\widetilde M| > |K|.
\]
\end{theorem}
\begin{proof}
Suppose $D = E * F$ for some group $H$ with $|H|<|G|$ and subgroups $E \leq GHT$ and $F \leq HGT$. By~\cite[Lemma 3.4]{R1}, we may assume that $p_1(F)=H$ and $k_1(F)=1$. 

Define a map $\psi : A \to H$ by setting $\psi(g,t)=h$ whenever $(h,g,t)\in F$. This map is well defined since $k_1(F)=1$, and it is surjective because $p_1(F)=H$.
Set $\widetilde{M} = \ker(\psi)$. By our assumption, $|H|<|G|$ and $|\widetilde{M}| = |A|/|H| = |G||K|/|H| > |K|$. For the first condition, let $(n,1)\in \widetilde{M}$ for some $n\in N$. Then \((1,n,1)\in F\). Since \((1,1,1)\in E\) and \(D=E*F\) by assumption, it follows that \((1,n,1)\in D\). But $D$ is a \text{covering} subgroup, hence we have $k_2(D)=1$, and therefore $n=1$. Thus $\widetilde{M}\cap (N\times 1)=1$.

For the converse, suppose there exists a normal subgroup $\widetilde{M}\unlhd A$ 
satisfying \( \widetilde{M}\cap (N \times 1)=1\). Set $L = A/\widetilde{M}$. We have $|L| = |A|/|\widetilde{M}| = |G|\ |K|/|\widetilde{M}| < |G|$, so we may look for a factorization over $L$.
Define subgroups $E\leq GLT$ and $F\leq LGT$ by
\begin{eqnarray*}
E &=& \{(\alpha(g,t),\psi(g,t),t)\mid (g,t)\in A\}\\
F &=& \{(\psi(g,t),g,t)\mid (g,t)\in A\},
\end{eqnarray*}
where $\psi$ is the canonical surjection from $A$ to $L$.

It is immediate that $D \subseteq E*F$. For the reverse inclusion, let $(\alpha(g_1,t),\psi(g_1,t),t)\in E$ and $(\psi(g_2,t),g_2,t)\in F$ such that $\psi(g_1,t)=\psi(g_2,t)$. Then $(\alpha(g_1,t),g_2,t)\in E*F$. Since $\widetilde{M}\cap (N\times 1)=1$, it follows that $g_1=g_2$, and $(g_2,t)\in A$. Therefore
\[
(\alpha(g_1,t),g_2,t)=(\alpha(g_2,t),g_2,t)\in D.
\]
Thus $E*F\subseteq D$, and the proof is complete.
\end{proof}
\begin{coro}\label{cor:invariance}
Let $D$ and $D'$ be covering subgroups of $GGT$.
\begin{enumerate}[leftmargin=2em]
  \item $D$ is $\ast$-factorable if and only if $D^{\mathrm{op}}$ is $\ast$-factorable.
  \item If $p_{2,3}(D)=p_{2,3}(D')$, or if $p_{1,3}(D)=p_{1,3}(D')$, then $D$ is $\ast$-factorable if and only if $D'$ is. In particular, $\ast$-factorability of a covering subgroup depends only on either of its invariants.
\end{enumerate}
\end{coro}
\begin{proof}
(1) If $D=E\ast F$ with $E\le GHT$ and $F\le HGT$, then $D^{\mathrm{op}}=F^{\mathrm{op}}\ast
E^{\mathrm{op}}$ over the same group $H$. (2) The criterion of Theorem~\ref{thm: description2}
depends only on $A=p_{2,3}(D)$ and its Goursat data, not on $\alpha$; this proves the first case.
The second case follows by applying the first to $D^{\mathrm{op}}$ and $D'^{\mathrm{op}}$, whose
common second invariant is $p_{1,3}(D)$, using~(1).
\end{proof}

\smallskip

As an application, we consider the case $K = \{1\}$.  Recall that the \emph{socle} $\Soc(G)$ of $G$ is the product of all minimal normal subgroups of $G$.
\begin{coro}\label{cor:socle}
Assume the conventions of Section \ref{section3 conventions}. Suppose $K = \{1\}$.  Then $D$ is not $\ast$-factorable if and only if
$\Soc(G) \leq N$.
\end{coro}
\begin{proof}
When $K = \{1\}$, the group $A$ is isomorphic to $G$ and $N\times 1$
corresponds to $N$.  A normal subgroup $\widetilde{M}\le G$ with $\widetilde{M} \cap N = \{1\}$ exists if and only if
some minimal normal subgroup of $G$ meets $N$ trivially, that is, $\Soc(G) \not\leq N$.
\end{proof}

We now give a reformulation of the above criterion when the covering subgroup $D\le GGT$ satisfies the condition that $K:=k_2(A)$ is central in $T_0:= p_2(A)$. In this case, there is a well-defined action of $G$ on $T_0$ by conjugation via $A$. Indeed for $g \in G$, choose any $(g, t_g) \in A$ and set $g \cdot x := t_g x t_g^{-1}$ for any $x\in T_0$.  Since two lifts of $g$ differ by an element of $K \leq Z(T_0)$, the action is well-defined.

Let $M \unlhd T_0$. A homomorphism $\theta\colon M \to G$ is called \emph{$G$-equivariant} if $\theta(t_g t t_g^{-1}) = g \theta(t)
g^{-1}$ for all $g \in G$ and $t \in M$ and for some (equivalently any) lift $t_g\in T_0$.

\begin{theorem}\label{thm: description3}
Assume the conventions of Section \ref{section3 conventions}. Additionally suppose $K\le Z(T_0)$. Then $D$ is $\ast$-factorable if and only if there exists a pair $(M, \theta)$ where $M \trianglelefteq T_0$ with $|M| > |K|$ and
$\theta\colon M \to G$ is a $G$-equivariant group homomorphism with          $\varphi(\theta(t)N) = tK$ for all $t \in M$.
\end{theorem}

\begin{proof} Assume \(D\) is $\ast$-factorable. By Theorem~\ref{thm: description2}, there exists a normal subgroup $\widetilde{M}\unlhd A$ with $\widetilde{M} \cap (N\times 1) = 1$. Set $M := p_2(\tilde M)$. The restriction $p_2|_{\tilde M}\colon \tilde M \to M$ is an isomorphism. Indeed it is surjective by definition, and its kernel is $\tilde M \cap (N\times 1) = 1$. In particular, we get $|M|>|K|$. Since
$\tilde M \trianglelefteq A$ and $p_{2}\colon A \to T_0$ is surjective, we have $M \trianglelefteq T_0$. Since $p_{2}|_{\tilde M}$ is an isomorphism, every $t \in M$ has a unique preimage $(\theta(t), t) \in \tilde M$. This induces a group homomorphism $\theta\colon M\to G$. The condition $\varphi(\theta(t)N) = tK$ holds since $(\theta(t),t) \in \tilde M \leq A$.  For $G$-equivariance,
conjugate $(\theta(t),t) \in \tilde M \trianglelefteq A$ by $(g,t_g) \in A$ to get $(g\theta(t)g^{-1}, t_g t t_g^{-1}) \in \tilde M$. By uniqueness of the preimage this gives $\theta(t_g t t_g^{-1}) = g\theta(t) g^{-1}$. 

\smallskip\noindent Conversely, given $(M, \theta)$ satisfying the conditions, set $\widetilde M := \{(\theta(t),t) : t \in M\}$. Then $\widetilde M \leq A$, and $\widetilde M \trianglelefteq A$ since for $(g',t') \in A$ and $(\theta(t),t) \in \widetilde M$, by the $G$-equivariance, we get $(g'\theta(t)g'^{-1}, t'tt'^{-1}) = (\theta(t'tt'^{-1}),t'tt'^{-1}) \in \widetilde M$.  Moreover $\widetilde M \cap (N\times 1) = 1$. Indeed if $(\theta(t),t) \in N\times 1$ then $t = 1$, so $\theta(t) = 1$.
\end{proof}

\smallskip 

Now we specialize to the case where $T$ is abelian. Assume the above notation. For any $T_0 \leq T$, the conjugation action of $G$ on $T_0$ via $A$ is trivial.  Hence $G$-equivariance of $\theta\colon M\to G$ reduces to requiring $\im(\theta) \leq Z(G)$.  Define $M_0 \leq T_0$ to be the unique subgroup with $K \leq M_0$ and $M_0/K = \varphi(Z(G)N/N)$.

Note that $M_0 = K$ if and only if $Z(G) \leq N$, and that any
valid $\theta\colon M'\to G$ in Theorem~\ref{thm: description3} must satisfy
$M' \leq M_0$ (since $\theta(m) \in Z(G)$ forces
$mK = \varphi(\theta(m)N) \in \varphi(Z(G)N/N)$, that is, $m \in M_0$).
Hence the search for a valid pair $(M', \theta)$ is reduced to subgroups of $M_0$.

\begin{coro}\label{cor:abelian_nonvanish}
Assume the notation from Section \ref{section3 conventions} and suppose $T$ is abelian. 
\begin{enumerate}[leftmargin=2em]
    \item If $Z(G) \leq N$, then $D$ is not $\ast$-factorable.
    \item If $Z(G)\not\le N$ (equivalently $M_0\supsetneq K$), then $D$ is $\ast$-factorable if and only if there exist a subgroup $M'\le M_0$ with $|M'|>|K|$ and a group homomorphism $\theta\colon M'\to Z(G)$ satisfying $\varphi(\theta(m)N)=mK$ for all $m\in M'$.
\end{enumerate}
\end{coro}

\begin{proof}
Theorem \ref{thm: description3} is applicable since $T_0$ is abelian, hence $D$ is $\ast$-factorable if and only if there exists $M\unlhd T_0$ with $|M| > |K|$ and a $G$-equivariant $\theta: M\to G$ with $\varphi(\theta(t)N) = tK$ for all $t\in M$. Since $T$ is abelian, every subgroup is normal, and the conjugation action of $G$ on $T_0$ is trivial, and hence the $G$-equivariance of $\theta$ is equivalent to $\im(\theta)\le Z(G)$. By the above observation, any such subgroup satisfies $M\le M_0$. Hence we get (2).

For (1), suppose $Z(G)\le N$ and let $(M', \theta')$ be such a pair. Then $\theta(m) \in Z(G)\le N$ implies $\varphi(\theta(m)N) = K$, so $mK = K$ for all $m\in M'$, that is, $M'\le K$, contradicting $|M'|>|K|$. Hence no such pair exists and $D$ is not $\ast$-factorable.
\end{proof}


\section{The Essential Algebra and its Idempotents}


Let $G$ and $T$ be fixed finite groups. After giving the formal definition of the essential algebra, we use the above criteria to extend Romero's coprime case description of the essential algebra. Then concentrate on the case where $T$ is abelian. Using Theorem~\ref{thm: description3} we give necessary and sufficient conditions for an image of a basis element of \(kB_T(G\times G)\) vanish in the essential algebra. We begin by recalling the definition of the essential algebra and then proceed to construct idempotents in this algebra and describe the structure of the algebra. 

We note that idempotents defined in this section are inspired by similar elements in the fibered Burnside algebras defined in \cite{BC} and they satisfy analogous properties. 

\begin{defi}
The \emph{essential algebra} $\widehat{kB_T}(G)$ of $kB_T$ at $G$ is 
\[
\widehat{kB_T}(G) := kB_T(G \times G) \Big/ \calI_T(G),
\]
where $\calI_T(G)$ denotes the ideal generated by all elements of the form
$\alpha \times_K^d \beta$ with $\alpha \in kB_T(G \times K)$, $\beta \in kB_T(K \times G)$, and $K$ ranging over all groups whose order is strictly less than $|G|$.
\end{defi}

\begin{nothing}\label{not: romero-generalT}
We will use the following facts from \cite{R1}, valid for an arbitrary finite group $T$ and an arbitrary commutative ring $k$. If $D\le GGT$ is not a covering subgroup, then $\bigl[\bigl[\tfrac{GGT}{D}\bigr]\bigr]=0$ in $\widehat{kB_T}(G)$ \cite{R1}. If a linear combination of pairwise distinct basis elements of $kB_T(G\times G)$ lies in $\calI_T(G)$, then every subgroup appearing with nonzero coefficient is $\ast$-factorable \cite[Lemma 3.2]{R1}; in particular, $\bigl[\bigl[\tfrac{GGT}{D}\bigr]\bigr]=0$ implies that $D$ is $\ast$-factorable, and the classes $\bigl[\bigl[\tfrac{GGT}{D}\bigr]\bigr]$, where $D$ runs over a set of representatives of the conjugacy classes of covering subgroups that are not $\ast$-factorable, are linearly independent in $\widehat{kB_T}(G)$ \cite{R1}. We also use the subgroups $D_{\sigma,\alpha,X}$ and the $k$-algebra $kB^{Z(G)}(T)\rtimes\Out(G)$ of \cite[Section 3]{R1}.
\end{nothing}

\begin{theorem}\label{thm: nosection}
Let $G$ and $T$ be finite groups such that no nontrivial quotient of $G$ is isomorphic to a section of $T$. The following statements hold.
\begin{enumerate}[leftmargin=2em]
  \item Every covering subgroup $D\le GGT$ with associated pair $(A,\alpha)$ and Goursat quintuple $(G,N,\varphi,K,T_0)$ satisfies $N=G$, that is, $A=G\times T_0$ and $D=D_{\sigma,\beta,T_0}$ for an automorphism $\sigma$ of $G$ and a homomorphism $\beta\colon T_0\to Z(G)$.
  \item No covering subgroup of $GGT$ is $\ast$-factorable.
  \item The classes $\big[\big[\tfrac{GGT}{D}\big]\big]$, where $D$ runs over a set of representatives of the conjugacy classes of covering subgroups, form a $k$-basis of $\widehat{kB_T}(G)$, and there is an isomorphism of $k$-algebras
        \[
        \widehat{kB_T}(G)\;\cong\;kB^{Z(G)}(T)\rtimes\Out(G).
        \]
\end{enumerate}
In particular, this holds whenever $\gcd(|G|,|T|)=1$, recovering \cite[Corollary 3.12]{R1}.
\end{theorem}
\begin{proof}
(1) By Goursat's lemma, $\varphi\colon G/N\xrightarrow{\ \sim\ }T_0/K$, and $T_0/K$ is a section of $T$. Then the hypothesis implies $G/N=1$, so $N=G$, $G\times 1\le A$, and $A=G\times T_0$ with $K=T_0$. Writing $\alpha(g,t)=\alpha(g,1)\alpha(1,t)$, the map $\sigma:=\alpha(-,1)$ is an automorphism of $G$ (its kernel is $\ker\alpha\cap(G\times 1)=1$) and $\beta:=\alpha(1,-)$ has central image, since $\beta(t)\sigma(g)=\sigma(g)\beta(t)$ for all $g\in G$, $t\in T_0$. Hence $D=D_{\sigma,\beta,T_0}$ in the notation of \cite[Section 3]{R1}.

(2) If $\widetilde M\trianglelefteq A$ satisfies $\widetilde M\cap(G\times 1)=1$, then $p_2$ embeds $\widetilde M$ into $T_0$, so $|\widetilde M|\le|T_0|=|K|$; by Theorem~\ref{thm: description2}, $D$ is not $\ast$-factorable.

(3) By (1) and (2), every covering subgroup is of the form $D_{\sigma,\beta,T_0}$ and none is $\ast$-factorable. By \ref{not: romero-generalT}, the non-covering classes vanish, so the covering classes span $\widehat{kB_T}(G)$, and they are linearly independent. Hence they form a $k$-basis. The injective $k$-algebra morphism $kB^{Z(G)}(T)\rtimes\Out(G)\to\widehat{kB_T}(G)$ of \cite[Theorem 3.11]{R1} has image the $k$-span of these classes, which is all of $\widehat{kB_T}(G)$. Hence it is an isomorphism.

For the last statement, if $\gcd(|G|,|T|)=1$ then the order of a common nontrivial quotient-section would divide both $|G|$ and $|T|$.
\end{proof}

\begin{nothing}
\label{sec: abelian}
For the remainder of the paper, we restrict our attention to the case where $T$ is a fixed abelian group such that $|T|\in k^{\times}$.
\end{nothing}

\begin{lem}\label{lem: product}
Let \(H\) be a finite group, \(E \leq GHT\), and \(F \leq HKT\). Set \(A := p_{2,3}(E)\)
and \(B := p_{1,3}(F).\) If \(p_2(E)=H\) or \(p_1(F)=H\), then
$$
\Big[\frac{GHT}{E} \Big] \times_H^d \Big[\frac{HKT}{F} \Big]
\;\cong\;
|T:k_2(AB)| \, \Big[\frac{GKT}{E*F} \Big].
$$
\end{lem}

\begin{proof}
By Lemma~\ref{Mackey}, the above product decomposes as a disjoint union indexed by a set of representatives of the double cosets $A \backslash (H \times T) / B$. If
\(p_1(A) = p_2(E)=H\) or \(p_1(B) = p_1(F)=H\), by Lemma~\ref{lem: Anormal}, either \(A\) or \(B\) is normal in \(H\times T\). Hence \(AB\) is a subgroup of \(H\times T\), and we have
$$
A \backslash (H\times T)/B = (H\times T)/(AB).
$$

Moreover every double coset in \((H\times T)/(AB)\) admits a representative of the form \((1,t)\) for some \(t\in T\) since $p_1(AB) = H$ under either hypothesis. Given $t, t'\in T$ the equality \((1,t)AB=(1,t')AB\) holds if and only if \( (1,t^{-1}t')\in AB,\) which is equivalent to \( t^{-1}t'\in k_2(AB).\) Therefore we obtain a bijection $T/k_2(AB) \to (H\times T)/AB.$
Representatives may be taken in the form $(1,t)$, and since $T$ is abelian, conjugation by $(1,1,t)$ fixes $F$; hence each star product $E\ast^{(1,1,t)}F$ equals $E\ast F$. The product is therefore a disjoint union of $|T:k_2(AB)|$ copies of $[GKT/E\ast F]$.
\end{proof}

\begin{notation}
Let $D$ be a subgroup of $GGT$. We denote by
\[
\Big[\Big[\frac{GGT}{D}\Big]\Big]
\]
the image of the transitive $GGT$-set $\Big[\frac{GGT}{D}\Big]$ in the essential algebra $\widehat{kB_T}(G)$ of $G$.

Since the set of transitive $GGT$-sets forms a $k$-basis of $kB_T(G \times G)$, their images form a generating set of the essential algebra.
\end{notation}

\begin{prop}
\label{zeros}
Let $D\le GGT$.  If $D$ is $\ast$-factorable, then
$\Big[\Big[\frac{GGT}{D}\Big]\Big]$
is equal to zero in $\widehat{kB_T}(G)$.
\end{prop}

\begin{proof}
By \cite[Lemma 3.4]{R1}, if $D = E * F$, then there exists a group $H_1$ with $|H_1| \leq |H|$ and subgroups $E_1 \leq GH_1T$ and $F_1 \leq H_1GT$ such that $D = E_1 * F_1$ and, moreover, $p_2(E_1) = H_1$.\\
Applying Lemma~\ref{lem: product}, we obtain
\begin{eqnarray*}
\Big[\frac{GGT}{D}\Big] 
= \Big[\frac{GGT}{E * F}\Big] 
&=& \Big[\frac{GGT}{E_1 * F_1}\Big]\\
&=& \frac{1}{|T : k_2(AB)|} \Big[\frac{GH_1T}{E_1}\Big]\times_{H_1}^d \Big[\frac{H_1GT}{F_1}\Big],
\end{eqnarray*}
where $A=p_{2,3}(E_1)$ and $B=p_{1,3}(F_1)$. Since $|H_1| < |G|$, the above expression factors through a group of strictly smaller order than $G$, and hence its image in the essential algebra is zero.
\end{proof}

By~\cite[Lemma 3.2]{R1}, if \(\Big[\Big[\frac{GGT}{D}\Big]\Big] = 0\), then \(D\) is $\ast$-factorable. Combining it with Proposition \ref{zeros}, we conclude that, since $T$ is abelian, \(\Big[\Big[\frac{GGT}{D}\Big]\Big] = 0\), if and only if \(D\) is $\ast$-factorable. In particular, $\calI_T(G)$ coincides with the $k$-span of the basis elements $[GGT/D]$ with $D$ $\ast$-factorable. Therefore, classes $\big[\big[\tfrac{GGT}{D}\big]\big]$, where $D$ runs over a set of representatives of conjugacy classes of covering subgroups that are not $\ast$-factorable form a $k$-basis of $\widehat{kB_T}(G)$.

\begin{nothing}
Let $A \leq G\times T$, and set $K = k_2(A)$. We define a subgroup $E_A \leq GGT$ by
\[
E_A := \{(g,g,t) \mid (g,t)\in A\},
\]
and an element $e_A$ in the endomorphism ring $kB_T(G\times G)$ by
\[
e_A := \frac{1}{|T:K|}\Big[\frac{GGT}{E_A}\Big].
\]
We denote by $\widehat{e}_A$ its image in the essential algebra.

It may occur that $\widehat{e}_A=0$. For example, if $p_1(A) \neq G$, then $\widehat{e}_A=0$. We call a subgroup $A \leq G\times T$ \emph{reduced} if $\widehat{e}_A \neq 0$, and denote by $\calR_{G\times T}$ the set of all reduced subgroups of $G\times T$. By Lemma~\ref{lem: Anormal}, $A$ is a normal subgroup of $G\times T$, when $p_1(A)=G$. Hence the set $\calR_{G\times T}$ is a subset of normal subgroups of $G\times T$. Moreover, the set $\calR_{G\times T}$ has a partial order via inclusion, and we denote it by $(\calR_{G\times T},\leq)$. 

As an immediate consequence of Theorem~\ref{thm: description3} and Corollary \ref{cor:abelian_nonvanish}, we get the following criterion for vanishing of $\widehat{e}_A$.
\end{nothing}

\begin{coro}
\label{desTabel}
Let $G$ be a group and let $A \leq G\times T$, with corresponding Goursat quintuple $(G, N, \varphi, K, T_0)$. Then 

\smallskip 

\noindent 1. $\widehat{e}_A=0$ if and only if there exists a subgroup $M \leq T_0$ with $|K| < |M|$ and a group homomorphism $\theta : M \to Z(G)$ such that
\[
\varphi(\theta(m)N) = mK \quad \text{for all } m \in M.
\]

\smallskip 

\noindent 2. If $Z(G) \leq N$, then $\widehat{e}_A \neq 0$.

\end{coro}

We also claim that the set $\calR_{G\times T}$ of reduced subgroups is \emph{upward} closed in $(\calR_{G\times T},\le)$: a subgroup containing a reduced subgroup is itself reduced. The following lemma is an equivalent statement.

\begin{lem}
\label{lem: topdown}
Let $B\leq A\leq G\times T$. If $\widehat{e}_A=0$ then $\widehat{e}_B=0$.
\end{lem}
\begin{proof}
If $p_1(A) \neq G$, then $p_1(B) \neq G$, and hence $\widehat{e}_B=0$, so the claim follows. Thus, we may assume that $p_1(A)=p_1(B)=G$. Let $(G,N,\varphi,K,T_0)$ and $(G,N',\varphi',K',T_0')$ be the Goursat quintuples corresponding to $A$ and $B$ respectively. By Proposition~\ref{zeros} and \cite[Lemma 3.2]{R1}, we know that $\widehat{e}_A = 0$ if and only if $E_A$ is $\ast$-factorable. By Theorem~\ref{thm: description2}, this is equivalent to the existence of normal subgroup $\widetilde{M}\unlhd A$ with $\widetilde{M}\cap (N\times 1) = 1$ and $|\widetilde{M}|>|K|$. Consider a normal subgroup $\widetilde{M}\cap B\unlhd B$. Since $N'\leq N$ we have $(\widetilde{M}\cap B)\cap (N'\times 1)\leq \widetilde{M}\cap (N\times 1)=1$. Moreover, 
\[\frac{|G||K'|}{|\widetilde{M}\cap B|}=\frac{|B|}{|\widetilde{M}\cap B|}=\frac{|B\widetilde{M}|}{|\widetilde{M}|} \leq \frac{|A|}{|\widetilde{M}|} = \frac{|G||K|}{|\widetilde{M}|}<|G|\]
and hence $|\widetilde{M}\cap B|>|K'|$. Thus $E_B$ is $\ast$-factorable, which implies $\widehat{e}_B=0$.
\end{proof}

Note that \(E_{G\times T}=\Delta(G)\times T\), and it is not $\ast$-factorable. Therefore, \(\widehat{e}_{G\times T}\) is a non-zero element of the essential algebra. Moreover, \(\widehat{e}_{G\times T}\) is the identity element of the essential algebra.
\begin{lem}
\label{ipots1}
Let \(G\) be a group and let \(A,B \in \calR_{G\times T}\). Then
$$
\widehat{e}_A\cdot \widehat{e}_B =
\begin{cases}
\widehat{e}_{A\cap B} & \text{if } A \cap B \in \calR_{G\times T},\\[6pt]
0 & \text{otherwise}.
\end{cases}
$$
\end{lem}

\begin{proof}
First observe that \((A,\mathrm{\pi_G})\) is the pair corresponding to the covering subgroup \(E_A\), where $\pi_G$ is the projection map from $A$ to $G$. Hence, by Lemma~\ref{lem: productofcovering},
\begin{eqnarray*}
E_A*E_B
&=&
\{(g,g,t)\mid (g,t)\in A \text{ and } (g,t)\in B\} \\
&=&
\{(g,g,t)\mid (g,t)\in A\cap B\} \\
&=&
E_{A\cap B}.
\end{eqnarray*}
Therefore, if \(A\cap B \notin \calR_{G\times T}\), then
\(
\widehat{e}_{A\cap B}=0
\)
and hence
\(
\widehat{e}_A\cdot \widehat{e}_B=0.
\)

Assume now that \(A\cap B \in \calR_{G\times T}\). We claim that $k_2(AB) = k_2(A)k_2(B)$. Indeed, it is clear that $k_2(AB) \supseteq k_2(A)k_2(B)$. For the converse, given $t\in k_2(AB)$, write $(1, t) = (a, s)(a^{-1}, u)$ for some $(a, s)\in A$ and $(a^{-1}, u)\in B$ with $t= su$. Since $A\cap B$ is reduced, we must have $p_1(A\cap B) = G$. Hence $(a, w)\in A\cap B$ for some $w\in T$. Then we get $sw^{-1}\in k_2(A)$ and $wu\in k_2(B)$. Thus $t = su = (sw^{-1})(wu)\in k_2(A)k_2(B)$.

Now by Lemma~\ref{lem: product}, the corresponding coefficient is
$$
\frac{|T : k_2(AB)|}{|T : k_2(A)|\,|T : k_2(B)|} = \frac{|T:k_2(A)k_2(B)|}{|T:k_2(A)||T:k_2(B)|} = 
\frac{1}{|T:k_2(A)\cap k_2(B)|} = \frac{1}{|T:k_2(A\cap B)|}.
$$
which is the coefficient that appear in \(\widehat{e}_{A\cap B}\), completing the proof.
\end{proof}

\begin{nothing}
Let \(G\) be a finite group and let \(D\) be a covering subgroup of \(GGT\). Denote by \(L\) and \(R\) the left and right invariants of \(D\), respectively, as defined in Section~\ref{sec: invariants}. Observe that \(E_R=D^{\mathrm{op}}*D\), and similarly \(E_L=D*D^{\mathrm{op}}\). These subgroups can be used to detect whether \(\Big[\Big[\frac{GGT}{D}\Big]\Big]\) vanishes in the essential algebra. More precisely, as we shall prove below, the elements \(\Big[\Big[\frac{GGT}{D}\Big]\Big]\), \(\widehat{e}_L\), and \(\widehat{e}_R\) either all vanish or are all non-zero.
\end{nothing}

\begin{lem}\label{lem: reduced-invariants}
Let $D\le GGT$ be a covering subgroup. Then

\smallskip

\noindent 1. 
\begin{equation*}
\label{eq: ipottimeselement}
\widehat{e}_{L_D}\cdot\Big[\Big[\frac{GGT}{D}\Big]\Big] =\,\Big[\Big[\frac{GGT}{D}\Big]\Big] =\Big[\Big[\frac{GGT}{D}\Big]\Big]\cdot\widehat{e}_{R_D},
\end{equation*}

\smallskip

\noindent 2. $\widehat e_{L_D} = 0$ if and only if $\widehat e_{R_D} = 0$ if and only if $\Big[\Big[\frac{GGT}{D}\Big]\Big] = 0$.
\end{lem}
\begin{proof} 1. Since $p_1(D)=G$, Lemma~\ref{lem: product} applies to $E_{L_D}$ and $D$. One checks
$E_{L_D}\ast D = D$, indeed, for $(g_1,g_2,t)\in D$ take $h=g_1$, so $(g_1,h,t)=(g_1,g_1,t)\in E_{L_D}$
(as $(g_1,t)\in L_D$) and $(h,g_2,t)\in D$. Conversely every element of $E_{L_D}\ast D$ has the
form $(h,g_2,t)\in D$. Hence
$[GGT/E_{L_D}]\times_G^d[GGT/D]=|T:k_2(L_D)|\,[GGT/D]$, and dividing by $|T:k_2(L_D)|$ gives
$\widehat e_{L_D}\cdot[[GGT/D]]=[[GGT/D]]$. The right-hand identity is symmetric.

    \smallskip 

\noindent 2. By part (1), it is clear that either $\widehat e_{L_D} = 0$ or $\widehat e_{R_D} = 0$ implies that $\Big[\Big[\frac{GGT}{D}\Big]\Big] = 0$. Conversely, if $\Big[\Big[\frac{GGT}{D}\Big]\Big] = 0$, then $D$ is $\ast$-factorable by \cite[Lemma 3.2]{R1}, say $D=E\ast F$ over $H$ with $|H|<|G|$. Then $E_{L_D} = E\ast (F\ast D^{op})$ and $E_{R_D} = (D^{op} \ast E) \ast F$ are $\ast$-factorable over $H$, so $\widehat e_{L_D} = \widehat e_{R_D} = 0$ by Proposition \ref{zeros}.
\end{proof}


\section{Central Idempotents of \(\widehat{kB_T}(G)\)}


In this section, we construct orthogonal central idempotents of the essential algebra using ideas inspired by \cite{BC}. The arguments closely follow those in~\cite[Section 4]{BC} and~\cite[Section 4]{CM}. For this reason, we omit most of the proofs and provide details only in the cases where the arguments differ from the original proofs. For the convenience of the reader, we indicate alongside each statement the corresponding result from~\cite{BC}, where a similar proof can be found.
\begin{nothing} \textbf{Orthogonal Idempotents.}
Using the construction given in \cite[Section 4.3]{BC}, we introduce orthogonal idempotents $\widehat{f}_A$ in the essential algebra $\widehat{kB_T}(G)$, where $A \in \calR_{G\times T}$.
We define
\begin{equation*}
\widehat{f}_A := \sum_{\substack{B \in \calR_{G\times T} \\ B \leq A}} \mu(B,A)\widehat{e}_B,
\end{equation*}
where $\mu$ denotes the Möbius function of the poset $\calR_{G\times T}$, which is equal to the M\"obius function $\mu_\unlhd$ of the poset of normal subgroups of $G\times T$ by Lemma~\ref{lem: topdown}. By the Möbius inversion formula, it follows that
\begin{equation*}
\widehat{e}_A = \sum_{\substack{B \in \calR_{G\times T} \\ B \leq A}} \widehat{f}_B,
\end{equation*}
for all $A \in \calR_{G\times T}$, and, in particular,
\begin{equation}
\label{eq: sumeqone}
1_{\widehat{kB_T}(G)} =\widehat{e}_{G\times T} = \sum_{B \in \calR_{G\times T}} \widehat{f}_B.
\end{equation}
\end{nothing}

\begin{prop}({\cite[Proposition 4.4]{BC}})
\label{prop: ortipots}
Let $G$ be a finite group and let $A,B \in \calR_{G\times T}$. Then
\begin{equation*}
\widehat{e}_A\cdot \widehat{f}_B = \widehat{f}_B\cdot \widehat{e}_A =
\begin{cases}
\widehat{f}_B & \text{if } B \leq A,\\
0 & \text{otherwise},
\end{cases}
\end{equation*}
and
\begin{equation*}
\widehat{f}_A\cdot \widehat{f}_B =
\begin{cases}
\widehat{f}_A & \text{if } A = B,\\
0 & \text{otherwise}.
\end{cases}
\end{equation*}
\end{prop}
\begin{proof}
The key observation is that if \(\widehat{e}_A=0\) and \(B\leq A\), then Lemma~\ref{lem: topdown} implies that \(\widehat{e}_C=0\) for every subgroup \(C\leq B\leq A\). Consequently, \(\widehat{f}_B=0\). The rest of the proof follows similarly to that of~\cite[Proposition 4.4]{BC}.
\end{proof}

Now we establish a theorem describing the interaction between the idempotents introduced above and the elements of the essential algebra. This result will allow us to define an equivalence relation on the poset of reduced subgroups such that the sum of the corresponding idempotents over each equivalence class yields a central idempotent of the essential algebra.

\begin{theorem}({\cite[Lemma 4.5]{BC}})
\label{thm: leftrightinvariants}
Let $G$ be a finite group and let $D$ be a covering subgroup of $GGT$ which is not $\ast$-factorable. Let $L$ and $R$ denote the left and right invariants of $D$, respectively. Then
\[
\widehat{f}_L\cdot \Big[\Big[\frac{GGT}{D}\Big]\Big] \cdot \widehat{f}_R
= \Big[\Big[\frac{GGT}{D}\Big]\Big] \cdot \widehat{f}_R
= \widehat{f}_L\cdot \Big[\Big[\frac{GGT}{D}\Big]\Big].
\]
\end{theorem}

\begin{proof}
It suffices to show that
\[
\widehat{e}_X\cdot \Big[\Big[\frac{GGT}{D}\Big]\Big] \cdot \widehat{f}_R = 0
\]
for all $X < L$ in $\calR_{G\times T}$. We have
\[
\widehat{e}_X\cdot \Big[\Big[\frac{GGT}{D}\Big]\Big]
= \frac{1}{|k_2(L):k_2(X)|}\Big[\Big[\frac{GGT}{E_X * D}\Big]\Big].
\]
Set $y := E_X * D$. If $y$ is not a covering subgroup of $GGT$, then the claim follows immediately. Hence, assume that $y$ is a covering subgroup.

By definition,
\begin{align*}
y &= \{(g,g',t)\mid \exists\, h \in G \text{ such that } (g,h,t)\in E_X \text{ and } (h,g',t)\in D\} \\
  &= \{(g,g',t)\mid (g,g,t)\in E_X \text{ and } (g,g',t)\in D\}.
\end{align*}
Hence $p_{2,3}(y)\subseteq R$. Moreover, as $X < L$, there exists $(g,t)\in L$ such that $(g,t)\notin X$. Also by definition of $L$, there exists $h \in G$ such that $(g,h,t)\in D$, and hence $(h,t)\in R$. 

We claim that $(h,t)\notin p_{2,3}(y)$. Indeed, if $(h,t)\in p_{2,3}(y)$, then there exists $g''\in G$ such that $(g'',h,t)\in D$ and $(g'',t)\in X$. Since $D$ is a covering subgroup, this would imply $g''=g$, and hence $(g,t)\in X$, a contradiction. Therefore $(h,t)\notin p_{2,3}(y)$, and so $p_{2,3}(y) \subsetneq R$.
It follows that the right invariant of $y$ is strictly contained in $R$, and hence
\[
\widehat{e_X}\cdot \Big[\Big[\frac{GGT}{D}\Big]\Big] \cdot \widehat{f}_R = 0
\]
by Proposition~\ref{prop: ortipots}.
\end{proof}\begin{nothing}
\textbf{Linkage.} For $A,B \in \calR_{G\times T}$, we say that $A$ and $B$ are \emph{linked}, written $A\sim B$, if there exists a covering subgroup $D \leq GGT$ such that $L_D = A$ and $R_D = B$. This is an equivalence relation; it is reflexive via $E_A$, it is symmetric via $D^{op}$ and the transitivity holds since $L_{D_1\ast D_2} = L_{D_1}$ and $R_{D_1\ast D_2} = R_{D_2}$ for covering subgroups $D_1$ and $D_2$ with $R_{D_1} = L_{D_2}$. Note that if $A, B$ are reduced and $D$ is covering with $L_D = A$, $R_D = B$, then by Lemma \ref{lem: reduced-invariants}, we have $[[GGT/D]]\neq 0$. We write $\calR_{G\times T}/\sim$ for a complete set of linkage class representatives. For each $A \in \calR_{G\times T}/\sim$, we write $\widetilde{A}$ for the equivalence class of $A$ in $\calR_{G\times T}$.
\end{nothing}

\begin{lem}
Let $G$ be a finite group. The partial order on $\calR_{G\times T}$ induces a partial order on the set of linkage classes $\calR_{G\times T}/\sim$, defined by
\[
\widetilde{A} \leq \widetilde{B}
\]
if there exist representatives $A' \sim A$ and $B' \sim B$ such that $A' \leq B'$. This relation is well defined and endows $\calR_{G\times T}/\sim$ with a partial order.
\end{lem}

\begin{proof}
    It is enough to prove that if $A\leq B$ then for any $B'\sim B$ there exists $A'\sim A$ such that $A'\leq B'$. For this we consider a covering group $D\leq GGT$ that links $B$ to $B'$, that is $p_{1,3}(D)=B$ and $p_{2,3}(D)=B'$. Lets consider
    \[
        E_A*D=\{(g_1,g_2,t)\in D\mid (g_1,t)\in A\}.
    \]
    Note that $p_{1,3}(E_A*D)=B\cap A=A$ and thus this subgroup links $A$ to the right invariant $A'$ of the group $E_A*D$. Moreover, we have $(E_A*D)*D^{op}=E_A*E_B=E_A$ and thus right invariant of the subgroup $E_A*D$ must be reduced and hence that $E_A*D$ is covering. Now by construction $A'\leq B'$ which completes the proof.
\end{proof}

\begin{nothing}
With this order we may form class sums of idempotents. For $A\in\calR_{G\times T}$ set
\[
\widetilde{e}_A:=\sum_{B\in\widetilde{A}}\widehat{e}_B
\qquad\text{and}\qquad
\widetilde{f}_A:=\sum_{B\in\widetilde{A}}\widehat{f}_B .
\]
\end{nothing}

\begin{prop}[cf.\ {\cite[Proposition 5.6]{BC}}]\label{prop: centipots}
Let $A,B\in\calR_{G\times T}$. Then
\[
\widetilde{e}_A\,\widetilde{f}_B=\widetilde{f}_B\,\widetilde{e}_A=0\ \text{ unless }\widetilde{B}\le\widetilde{A},
\qquad
\widetilde{e}_A\,\widetilde{f}_A=\widetilde{f}_A\,\widetilde{e}_A=\widetilde{f}_A,
\]
\[
\widetilde{f}_A\,\widetilde{f}_B=
\begin{cases}\widetilde{f}_A,&\widetilde{A}=\widetilde{B},\\[2pt]0,&\text{otherwise.}\end{cases}
\]
\end{prop}
\begin{proof}
Expand the class sums and apply Proposition~\ref{prop: ortipots} term-wise. The only point beyond \cite[Proposition 5.6]{BC} is that linked subgroups have equal order: if $A\sim B$ then $|A|=|B|$,
since a covering $D$ with $L_D=A$, $R_D=B$ gives a bijection $R_D\to L_D$. Hence within a class $\widetilde A$, $A'\le A''$ forces $A'=A''$, which gives $\widetilde e_A\widetilde f_A=\widetilde f_A$.
\end{proof}

\begin{lem}[cf.\ {\cite[Lemma 6.3]{BC}}]
\label{lem: f-orthogonality}
Let $D$ be a covering subgroup of $GGT$, and $A,B\in\calR_{G\times T}$.
If $\widetilde{f}_A\,\Big[\Big[\tfrac{GGT}{D}\Big]\Big]\,\widetilde{f}_B\neq 0$, then $\widetilde{A}=\widetilde{B}$.
\end{lem}
\begin{proof}
If $\widetilde{f}_A\,[[\tfrac{GGT}{D}]]\,\widetilde{f}_B\neq 0$, there exist $A'\sim A$ and
$B'\sim B$ with $\widehat{f}_{A'}\,[[\tfrac{GGT}{D}]]\,\widehat{f}_{B'}\neq 0$. Since
$\widehat{f}_{B'}=\sum_{B''\le B'}\mu(B'',B')\widehat{e}_{B''}$, there is some
$B''\le B'$ in $\calR_{G\times T}$ with
$\widehat{f}_{A'}\,[[\tfrac{GGT}{D}]]\,\widehat{e}_{B''}\neq 0$; in particular
$[[\tfrac{GGT}{D}]]\,\widehat{e}_{B''}\neq 0$. Note that
$D\ast E_{B''}=\{(g_1,g_2,t)\in D\mid (g_2,t)\in B''\}$, so
$A'':=p_{1,3}(D\ast E_{B''})\le p_{1,3}(D)$ and
$p_{2,3}(D\ast E_{B''})\le B''\le B'$.
Now we get  
\[
0 \neq \widehat{f}_{A'}\,\Big[\Big[\tfrac{GGT}{D}\Big]\Big]\,\widehat{e}_{B^{''}} = \widehat{f}_{A'}\widehat{e}_{A^{''}}\,\Big[\Big[\tfrac{GGT}{D}\Big]\Big]\,\widehat{e}_{B^{''}}
\]
which implies $A'\le A^{''}$ by the above proposition. Therefore we get
\[
\widetilde{A} = \widetilde{A'} \le \widetilde{A^{''}} = \widetilde{p_{1,3}(D\ast E_{B^{''}})}= \widetilde{p_{2,3}(D\ast E_{B^{''}})}\le \widetilde{B^{''}}\le \widetilde{B'} = \widetilde{B}.
\]
By symmetry, we get the reverse relation and hence the equality, as required.
\end{proof}

\begin{theorem}[cf.\ {\cite[Corollary 6.4]{BC}}]
\label{thm: central ipots}
Let $G$ be a finite group.
\begin{enumerate}
  \item The elements $\widetilde{f}_A$, $A\in\calR_{G\times T}/\sim$, are mutually orthogonal
        central idempotents of $\widehat{kB_T}(G)$ whose sum is $1$.
  \item Consequently $\widehat{kB_T}(G)$ decomposes into two-sided ideals
        \begin{equation}\label{eq: toideals}
          \widehat{kB_T}(G)=\bigoplus_{A\in\calR_{G\times T}/\sim}\widetilde{f}_A\,\widehat{kB_T}(G).
        \end{equation}
  \item $\widehat{kB_T}(G)$ also decomposes into $k$-submodules indexed by linkage classes,
        \begin{equation}\label{eq: tosubmodules}
          \widehat{kB_T}(G)=\bigoplus_{A\in\calR_{G\times T}/\sim}\widehat{kB_T}(G)^{\widetilde{A}},
        \end{equation}
        where $\widehat{kB_T}(G)^{\widetilde{A}}$ is spanned by the
        $\big[\big[\tfrac{GGT}{D}\big]\big]$ with $L_D\in\widetilde{A}$ (equivalently $R_D\in\widetilde{A}$).
\end{enumerate}
\end{theorem}
\begin{proof}
(1) Orthogonality and idempotency are Proposition~\ref{prop: centipots}, and
$\sum_A\widetilde{f}_A=1$ is Equation~(\ref{eq: sumeqone}); centrality follows from
Lemma~\ref{lem: f-orthogonality} via the three-decomposition argument of {\cite[Corollary 6.4]{BC}}.
Second part is immediate from~(1). The third part is the decomposition by left (equivalently right) invariants; its relation to Equation~(\ref{eq: toideals}) is established in Lemma~\ref{lem: comparing}.
\end{proof}

Now we define groups \(\Gamma_A\) and bisets \({}_A\Gamma_B\) associated to reduced subgroups \(A,B\in \mathcal{R}_{G\times T}\). We show that the elements of these bisets induce group isomorphisms and consequently equivalences between the corresponding module categories of group algebras. These groups will play a fundamental role in the next section in the description of the essential algebra and its simple modules. We also prove that \(\Gamma_A\) is isomorphic to the group of outer automorphisms of \(A\) that fix the second component.
\begin{defi}
Let $G$ be a finite group and let $A,B\in \calR_{G\times T}$ be linked. Set $K:=k_2(A)$. Define 

\smallskip

\[
\Gamma_A = \left\{ \frac{1}{|T:K|}\,\Big[\Big[\frac{GGT}{D}\Big]\Big] \;\middle|\; D \leq GGT \text{ such that } L_D = R_D = A \right\},
\]

\[
{}_A\Gamma_B = \left\{\Big[\Big[\frac{GGT}{D}\Big]\Big] \;\middle|\; D \leq GGT \text{ such that } L_D = A \text{ and } R_D = B \right\}.
\]
\end{defi}
\begin{prop}
Let $G$ be a group and let $A,B\in \calR_{G\times T}$ be linked. Then the following statements hold:
\begin{enumerate}
    \item $(\Gamma_A,\times_G^d)$ is a group with identity element $\widehat{e}_A$.
      \item The set $_A\Gamma_B$ is a $(\Gamma_A, \Gamma_B)$-biset which is both left and right transitive, as well as free on both sides.
      \item Any element $\gamma \in {}_A\Gamma_B$ induces an isomorphism of groups
    \[
    \gamma : \Gamma_B \xrightarrow{\sim} \Gamma_A,
    \]
    given by the map sending $x\in \Gamma_B$ to $\frac{1}{|T:K||T:L|}\gamma\cdot x\cdot\gamma^{op}$, where $K:=k_2(A)$ and $L:=k_2(B)$.
        \item The functor
    \[
    k\big[ {}_A\Gamma_B\big] \otimes_{k\Gamma_B} - 
    : k\Gamma_B\text{-mod} \longrightarrow k\Gamma_A\text{-mod}
    \]
    is an equivalence of categories. In particular, it yields a canonical bijection
    \[
    \mathrm{Irr}(k\Gamma_B) \xrightarrow{\sim} \mathrm{Irr}(k\Gamma_A).
    \]
       \item There is an isomorphism of groups
\(
(\Gamma_A,\times_{G}^{d}) \longrightarrow \Out_T(A),
\)
where $\Out_T(A) = \Aut_T(A)/ \Inn(A)$ is the group of outer automorphisms of \(A\) fixing the second component.
\end{enumerate}
\end{prop}
\begin{proof}
The proofs of statements (2), (3), (4) follow similarly to those in~\cite[Section 6.1]{BC} with necessary modifications. We provide proofs for (1) and (5).

\smallskip

\noindent 1. First $E_A * D = D$ for any covering subgroup $D$ with $L_D = R_D = A$, and hence $\widehat{e}_A$ is the identity element of $\Gamma_A$.

Let $E$ and $F$ be covering subgroups of $GGT$ such that $L_E = R_E = L_F = R_F = A$. Then
\[
E*F = \{(g_1,g_2,t)\mid \exists\, h \in G \text{ such that } (g_1,h,t)\in E \text{ and } (h,g_2,t)\in F\}.
\]
It is clear that $L_{E*F} \leq A$. Conversely, let $(g,t)\in A$. Then there exists $h \in G$ such that $(g,h,t)\in E$, hence $(h,t)\in A$, and therefore $(h,g',t)\in F$ for some $g'\in G$. This implies $(g,g',t)\in E*F$, so $(g,t)\in L_{E*F}$. Thus $L_{E*F} = A$, and similarly $R_{E*F} = A$. Hence $\Gamma_A$ is closed under $\times_G^d$.

Moreover, for any such $D$, one has $D * D^{op} = E_A$, so every element admits an inverse. Finally, since by the Lemma~\ref{lem: product} we have   
\[
\Big[\Big[\frac{GGT}{E}\Big]\Big]\times_G^d \Big[\Big[\frac{GGT}{F}\Big]\Big]
= |T:K| \cdot \Big[\Big[\frac{GGT}{E*F}\Big]\Big],
\]
the coefficients can be normalized by $\frac{1}{|T:K|}$, and the group structure follows.

\smallskip

\noindent 5. Let
\(
\frac{1}{|T:K|}\Big[\Big[\frac{GGT}{D}\Big]\Big]
\)
be an element of \(\Gamma_A\). Define a map
\(
\sigma_D \colon A \to A
\)
by setting
\(
\sigma_D(h,t)=(g,t)
\)
whenever \((g,h,t)\in D\).

We first show that \(\sigma_D\) is well defined. Suppose that
\(
(g,h,t),(g',h,t)\in D.
\)
Then
\(
(g^{-1}g',1,1)\in D.
\)
Since \(A\) is reduced, the subgroup \(D\) is covering, and therefore \(g=g'\). Hence \(\sigma\) is well defined. By construction, \(\sigma_D\) fixes the second component. Moreover it is straightforward to show that $\sigma_D$ is an automorphism and the assignment $D\mapsto \sigma_D$ is a group homomorphism. 

Conversely, let \(\sigma\) be an automorphism of \(A\) fixing the second component. Define $D_{\sigma}=\{(g,h,t)\mid \sigma(h,t)=(g,t)\}$.
Then clearly \(D_{\sigma}\) is a covering subgroup of \(GGT\) with $L_D = R_D = A$.

Next, observe that two automorphisms determine conjugate subgroups of \(GGT\) whenever they differ by conjugation with an element of \(G\times T\). Since \(p_1(A)=G\), the automorphisms of \(A\) induced by conjugation with elements of \(G\times T\) are precisely the inner automorphisms of \(A\). Consequently, the above correspondence induces an isomorphism $\Gamma_A \cong \Out_T(A)$.
\end{proof}


\section{Simple modules of the Essential Algebra}


In this section we parametrize the simple modules of $\widehat{kB_T}(G)$. Following the strategy of \cite[Section 9]{BC} and \cite[Section 4]{CM}, we first relate the two decompositions of $\widehat{kB_T}(G)$ obtained in Section 5 by means of an auxiliary lemma, and then deduce the parametrization from Theorem \ref{thm: matrixalg}.

\begin{nothing}
Let $A\in\calR_{G\times T}$ and enumerate $\widetilde{A}=\{A_1,\dots,A_n\}$. By definition $\widehat{kB_T}(G)^{\widetilde{A}}$ is spanned by $\big[\big[\tfrac{GGT}{D}\big]\big]$ with
$L_D,R_D\in\widetilde{A}$. Writing $(L_D,R_D)=(A_i,A_j)$ gives the decomposition into $k$-submodules
\begin{equation}\label{eq: directsum}
  \widehat{kB_T}(G)^{\widetilde{A}}=\bigoplus_{i,j=1}^{n}k[{}_{A_i}\Gamma_{A_j}].
\end{equation}
\end{nothing}

\begin{lem}[cf.\ {\cite[Lemma 6.6]{BC}}, {\cite[Lemma 4.9]{CM}}]
\label{lem: comparing}
Let $G$ be a group and $A\in\calR_{G\times T}$. Write $\widetilde{A}=\{A_1,\dots,A_n\}$.
\begin{enumerate}
  \item[(a)] If $D$ is covering with $R_D=A$ and $B\in\calR_{G\times T}$ with $B\not\le A$, then $\big[\big[\tfrac{GGT}{D}\big]\big]\,\widehat{f}_B=0$.
  \item[(b)] $\displaystyle\bigoplus_{\widetilde{B}\le\widetilde{A}}\widehat{kB_T}(G)^{\widetilde{B}}
    =\bigoplus_{\widetilde{B}\le \widetilde{A}}\widehat{kB_T}(G)\,\widetilde{f}_B.$
  \item[(c)] The projection $\omega:\widehat{kB_T}(G)^{\widetilde{A}}\to\widehat{kB_T}(G)\,\widetilde{f}_A$, $b\mapsto b\widetilde{f}_A$, with respect to Equation~(\ref{eq: toideals}) of Theorem \ref{thm: central ipots} is a $k$-module isomorphism, with inverse the projection with respect to Equation~(\ref{eq: tosubmodules}) of Theorem \ref{thm: central ipots}.
 \item[(d)] The map $\omega$ is the direct sum, over the decomposition in Equation~(\ref{eq: directsum}), of the $k$-module isomorphisms
\[
\omega_{ij}:k[{}_{A_i}\Gamma_{A_j}]\to\widehat{f}_{A_i}\,\widehat{kB_T}(G)\,\widehat{f}_{A_j}, \qquad b\mapsto\widehat{f}_{A_i}\,b\,\widehat{f}_{A_j}.
\]
  \item[(e)] For $b_{ij}\in k[{}_{A_i}\Gamma_{A_j}]$ and $b_{lm}\in k[{}_{A_l}\Gamma_{A_m}]$,
             \[
               \omega_{ij}(b_{ij})\,\omega_{lm}(b_{lm})=
               \begin{cases}\omega_{im}(b_{ij}b_{lm}),&j=l,\\[2pt]0,&j\neq l.\end{cases}
             \]
  \item[(f)] The correspondence $k\Gamma_A\to\widehat{f}_A\,\widehat{kB_T}(G)\,\widehat{f}_A$ given by  $a\mapsto\widehat{f}_A\,a\,\widehat{f}_A$, is a $k$-algebra isomorphism.
\end{enumerate}
\end{lem}
\begin{proof}
(a) We have $[[GGT/D]]\widehat{f}_B = [[GGT/D]]\widehat{e}_A\widehat{f}_B = 0$ by Lemma \ref{lem: reduced-invariants} and Proposition \ref{prop: ortipots}. The remaining parts are proved as in \cite[Lemma 6.6]{BC} and \cite[Lemma 4.9]{CM}.
\end{proof}

\begin{theorem}{\cite[Theorem 6.2]{BC}}
\label{thm: matrixalg}
Let $G$ be a finite group and $T$ be a finite abelian group. Let $k$ be a commutative unitary ring such that $|T|\in k^{\times}$. Then there is a $k$-algebra isomorphism
\[
    \bigoplus_{\widetilde{A}\in \calR_{G\times T}/\sim}\Mat_{|\widetilde{A}|}(k\Gamma_A)\xlongrightarrow{\sim} \widehat{kB_T}(G).
\]
Moreover, for every linkage class $\widetilde{A}=\{A_1,A_2,\cdots A_n\}\in \calR_{G\times T}/\sim$ and for each $i = 1,2, . . . , |\widetilde{A}|$, the above isomorphism sends the idempotent matrix
\[
e_i = \text{diag}(0,\ldots , 0, 1, 0,\ldots , 0) \in \Mat_{|\widetilde{A}|}(k\Gamma_A)
\]
in the $\widetilde{A}$-component to the idempotent $\widehat{f}_{A_i}\in \widehat{kB_T}(G)$.
\end{theorem}
\begin{proof}
By Equation~(\ref{eq: toideals}) it suffices to construct a $k$-algebra isomorphism $\Mat_n(k\Gamma_A)\xrightarrow{\sim}\widehat{kB_T}(G)\,\widetilde{f}_A$ sending $e_i$ to $\widehat{f}_{A_i}$ for a fixed $\widetilde{A}\in\calR_{G\times T}$.

For each $i\in \{1,2,\ldots, n\}$, choose $D_i\le GGT$ with $\big[\big[\tfrac{GGT}{D_i}\big]\big]\neq0$, $L_{D_i}=A$, $R_{D_i}=A_i$,
set $x_i:=\big[\big[\tfrac{GGT}{D_i}\big]\big]$, and put
\[
  y_i:=\tfrac{1}{|T:k_2(A_i)|}\,x_i,\qquad \bar y_i:=\tfrac{1}{|T:k_2(A)|}\,x_i^{op}.
\]
By Lemma~\ref{lem: product},
\begin{equation}\label{eq: yybar}
  y_i\times_G^d \bar y_i=\widehat{e}_A,\qquad \bar y_i\times_G^d y_i=\widehat{e}_{A_i}.
\end{equation}
Define $\sigma_{ij}:k\Gamma_A\to k[{}_{A_i}\Gamma_{A_j}]$, $a\mapsto\bar y_i\,a\,y_j$; by Lemma~\ref{lem: product}
this lands in $k[{}_{A_i}\Gamma_{A_j}]$, and $b\mapsto y_i\,b\,\bar y_j$ is its inverse by Equation~(\ref{eq: yybar}),
since $\widehat{e}_A$ is the identity of $\Gamma_A$. Moreover, for $a,a'\in k\Gamma_A$,
\begin{equation}\label{eq: sigmamult}
  \sigma_{ij}(a)\,\sigma_{jl}(a')=\bar y_i\,a\,(y_j\,\bar y_j)\,a'\,y_l
  =\bar y_i\,a\,\widehat{e}_A\,a'\,y_l=\bar y_i\,(aa')\,y_l=\sigma_{il}(aa'),
\end{equation}
using $y_j\times_G^d\bar y_j=\widehat{e}_A$. Taking the direct sum over $i,j$ and using Equation~(\ref{eq: directsum}) yields a $k$-module isomorphism $\sigma:\Mat_n(k\Gamma_A)\xrightarrow{\sim}\widehat{kB_T}(G)^{\widetilde{A}}$.

Let $\omega,\omega_{ij}$ be as in Lemma~\ref{lem: comparing}. Then
$\rho:=\omega\circ\sigma:(a_{ij})\mapsto\sum_{i,j}\omega_{ij}(\sigma_{ij}(a_{ij}))$ is a $k$-module isomorphism
onto $\widehat{kB_T}(G)\,\widetilde{f}_A$. Combining Lemma~\ref{lem: comparing}(e) with Equation~(\ref{eq: sigmamult}),
\begin{eqnarray*}
  \rho\big((a_{ij})\big)\,\rho\big((a'_{lm})\big)
  &=&\sum_{i,j,m}\omega_{im}\!\big(\sigma_{ij}(a_{ij})\,\sigma_{jm}(a'_{jm})\big)\\
  &=&\sum_{i,m}\omega_{im}\!\Big(\sigma_{im}\big(\textstyle\sum_j a_{ij}a'_{jm}\big)\Big)\\
  &=&\rho\big((a_{ij})(a'_{lm})\big),
\end{eqnarray*}
so $\rho$ is a $k$-algebra isomorphism. Finally
$\rho(e_i)=\omega_{ii}(\sigma_{ii}(\widehat{e}_A))=\widehat{f}_{A_i}\,\widehat{e}_{A_i}\,\widehat{f}_{A_i}=\widehat{f}_{A_i}$.
\end{proof}

\begin{theorem}[cf. \ {\cite[Section 9.1]{BC}}]\label{thm: simples}
Let $G$ be a finite group and $T$ be a finite abelian group. Let $k$ be a commutative unitary ring such that $|T|\in k^{\times}$. There is a bijective correspondence between the isomorphism classes of simple $\widehat{kB_T}(G)$-modules and the set of pairs $(A,[V])$, where $A$ runs over representatives of linkage classes of reduced elements in $\calR_{G\times T}$, and $[V]$ runs over isomorphism classes of simple $k\Gamma_A$-modules. \\
This correspondence is given by associating to $(A,[V])$ the $\widehat{kB_T}(G)$-module
    \[
   \widehat{kB_T}(G)\,\widehat{f}_A \otimes_{k\Gamma_A} V.
    \]
\end{theorem}
\begin{proof}
By Equation~(\ref{eq: toideals}), every simple module of the essential algebra belongs to exactly one block
and hence the set of isomorphism classes of simple modules is in bijection with the set of pairs $(A,[V])$, where $A\in \calR_{G\times T}/\sim$ and $[V]$ is the isomorphism class of a simple module over the algebra $\widehat{kB_T}(G)\widetilde{f}_A$. By Theorem~\ref{thm: matrixalg}, the composition $\omega\circ\sigma$ yields an isomorphism of $k$-algebras
\[
\Mat_{|\widetilde{A}|}(k\Gamma_A)
\cong
\widehat{kB_T}(G)\widetilde{f}_A.
\]
Therefore, it induces a bijection
\[
\Irr(\Mat_{|\widetilde{A}|}(k\Gamma_A))
\longrightarrow
\Irr(\widehat{kB_T}(G)\widetilde{f}_A).
\]
Since a matrix algebra is Morita equivalent to its coefficient algebra, the simple modules of $\Mat_{|\widetilde{A}|}(k\Gamma_A)$ are naturally in bijection with the simple modules of $k\Gamma_A$. Let $A\in\calR_{G\times T}$ be a representative of its linkage class and let $[V]\in\Irr(k\Gamma_A)$. Writing $A_i=A$ in the enumeration of the members of $\widetilde{A}$ used in Theorem~\ref{thm: matrixalg}, let $e_i=\text{diag}(0,\ldots,0,1,0,\ldots,0)$. Identifying $k\Gamma_A$ with the corner algebra $e_i\Mat_{|\widetilde{A}|}(k\Gamma_A)e_i$ via the map $a\mapsto e_iae_i$, the Morita equivalence is given by tensoring with the bimodule $\Mat_{|\widetilde{A}|}(k\Gamma_A)e_i$. Hence the simple module $V$ induces to the simple $\Mat_{|\widetilde{A}|}(k\Gamma_A)$-module $\Mat_{|\widetilde{A}|}(k\Gamma_A)\otimes_{k\Gamma_A}V$. Finally, choosing $D_i=E_A$ in the construction of the isomorphism $\omega\circ\sigma$ from the proof of Theorem~\ref{thm: matrixalg}, this simple module is transported to the simple module $\widetilde{f}_A\widehat{kB_T}(G)\widetilde{f}_A \otimes_{k\Gamma_A}V$. Since $\widetilde{f}_A$ is a central idempotent of $\widehat{kB_T}(G)$, this is equal to $\widehat{kB_T}(G)\widetilde{f}_A
\otimes_{k\Gamma_A}V = \widehat{kB_T}(G)\widehat{f}_A
\otimes_{k\Gamma_A}V$, which proves the stated parametrization.
\end{proof}

\section{On the group of $T$-fixing automorphisms}
In this section, we describe the group $\Aut_T(G)$ as an extension group and completely determine the structure of $\Gamma_A$ in a special case. We still assume $G$ is a fixed finite group and $T$ is a fixed finite abelian group. 

\begin{nothing}
  (i)  Let $A\in \calR_{G\times T}$ with Goursat quintuple $(G, N, \varphi, K, T_0)$. Since $N$ is normal, $G$ acts on $N$ via conjugation, denoted by $\rho$. This induces a homomorphism $\bar\rho: G/N \to \Out(N)$. 

  \noindent (ii) For each $t\in T_0$, we fix an element $g_t\in G$ such that $(g_t, t)\in A$ and set $g_1 = 1$. Then any other choice of $g\in G$ with $(g, t)\in A$ satisfies $g = g_t n$ for some $n\in N$. We let $T_0$ act on $Z(N)$ on the right via $z^t = g_t^{-1}zg_t$. Clearly this is independent of choice of $g_t$ and its restriction to $K$ is trivial.

  \noindent (iii) The above action of $T_0/K$ coincides with the action through the composition $\bar\rho\circ\varphi^{-1}$. We write $Z^1(T_0, Z(N))$ for the group of 1-cocycles with respect to this action. 
\end{nothing}

\begin{prop}\label{prop: autta}
    Let $A\in\calR_{G\times T}$ with the corresponding Goursat data $(G, N, \varphi, K, T_0)$ and let $W$ denote the image of the restriction map $\Aut_T(A)\to \Aut(N)$. Then there is an exact sequence of groups
    \[
    1\longrightarrow Z^1\bigl(T_0,Z(N)\bigr)\longrightarrow\operatorname{Aut}_T(A) \xrightarrow{\ \Res\ }W\longrightarrow 1,
    \]
    where the first map sends a 1-cocycle $\gamma$ to the automorphism $\theta_\gamma(g, t) = (g\gamma(t), t)$.
\end{prop}
\begin{proof}
    The proof consists of two steps. First we show that restriction of an automorphism in $\Aut_T(A)$ produces a 1-cocycle of $A$ in $N$ for the conjugation action of $A$ and then show that such an automorphism is in the kernel of the restriction if and only if the induced 1-cocycle lies in the group $Z^1\bigl(T_0,Z(N))$. 

    \noindent {\bf Step 1:} Let $\theta\in \Aut_T(A)$. We write $\theta(g,t) = (\theta_1(g,t), t)$ with $\theta_1(g,t):= p_1(\theta(g,t))$. Hence the restriction can be described as $\Res : \Aut_T(A)\to \Aut(N),$ $\theta\mapsto \theta_1(- , 1)$. Also, since $(g,t)^{-1}(\theta(g,t)) = 
    (g^{-1}\theta_1(g,t), 1)\in N\times 1$, we may write
    \[
    \theta(g,t) = (g c_\theta(g,t), t)
    \]
    with $c_\theta(g,t) = g^{-1}\theta_1(g,t)\in N$. Hence we obtain a function $c_\theta: A\to N$. To see that this is a crossed homomorphism, let $(g,t), (g', t')\in A$. Then comparing $\theta(gg',tt')$ and $\theta(g,t)\theta(g',t')$, we obtain
    \begin{equation}\label{eq: c-theta}
    c_\theta(gg', tt') = c_\theta(g,t)^{g'}\cdot c_\theta(g',t'),
    \end{equation}
    as required. Notice that $c_\theta \equiv 1$ if and only of $\theta = \id_A$. 

    \noindent {\bf Step 2:} Now assume $\theta\in\ker(\Res)$ and set $c:= c_\theta$. We have $c(n,1)= 1$ for all $n\in N$. For any $(g,t)\in A$ and $n\in N$, we have $(n,1)(g,t) = (g,t)(n^g, 1)$. Applying $c$ to both sides and using Equation (\ref{eq: c-theta}), we get $c(g,t) = c(g,t)^{n^g}$. Since this holds for all $n\in N$, we must have $c(g,t)\in Z(N)$ for all $(g, t)\in A$. Moreover if $(g,t), (g',t)\in A$, then $g' = gn$ for some $n\in N$ and hence $c(g',t) = c(gn, t) = c(g,t)^{n} = c(g,t)$ by Equation (\ref{eq: c-theta}). In particular, $c(g,t)$ depends only on $t$ and not the particular choice of $g$ with $(g,t)\in A$. Hence we obtain a map $\gamma: T_0\to Z(N)$ given by $\gamma(t) := c(g,t)$ for some (and hence any) $g\in G$ with $(g,t)\in A$. Finally, Equation (\ref{eq: c-theta}) becomes 
    \[
    \gamma(tt') = \gamma(t)^{t'}\cdot \gamma(t').
    \]
    Hence $\gamma\in Z^1(T_0, Z(N))$. Conversely given $\gamma\in Z^1(T_0, Z(N))$, we define $\theta_\gamma:A\to A$ by $\theta_\gamma(g,t) := (g\gamma(t), t)$. It is easy to check that $\theta_\gamma$ is in $\Aut_T(A)$ whose restriction to to $N\times 1$ is identity. Moreover $\theta_{\gamma\gamma'} = \theta_\gamma\theta_{\gamma'}$ for any $\gamma, \gamma'\in Z^1(T_0, Z(N))$. Hence we obtain an isomorphism $Z^1(T_0, Z(N)) \cong\ker(\Res)$, given by $\gamma\mapsto \theta_\gamma$, as required.
\end{proof}
\begin{nothing}
    Notice that every inner automorphism of $A$ is $T$-fixing and hence we have $\Inn(A)\le\Aut_T(A)$. Hence once the group $\Aut_T(A)$ is determined, one may further evaluate the quotient by the inner automorphisms to obtain the group $\Gamma_A$. 

    As a corollary of the above result, we get explicit descriptions on two special cases.
\end{nothing}
\begin{coro}
    Let $K\le T$ and $A = G\times K$. The $A$ is reduced, its linkage class is $\widetilde A = \{A\}$ and there is an isomorphism of groups
    \[
    \Gamma_{G\times K} \cong \Out(G)\ltimes \Hom(K, Z(G))
    \]
    where $\Out(G)$ acts on $\Hom(K, Z(G))$ by post-composition.
\end{coro}
\begin{proof}
The Goursat quintuple of $A = G\times K$ is $(G, G, \id, K, K)$. Hence $A$ is reduced by Corollary \ref{desTabel}. To determine its linkage class, let $D\le GGT$ be a covering subgroup with $R_D = A$ and the corresponding pair $(A, \alpha)$. Since $\ker\alpha\cap (G\times 1) = 1$, the induced homomorphism $\sigma := \alpha(- ,1)$ is an automorphism of $G$. Also the homomorphism $\beta:= \alpha(1, -): K\to G$ has image in $Z(G)$ since $\sigma(g)\beta(t) = \beta(t)\sigma(g)$ for all $g\in G$ and $t\in K$. Thus we have 
\[
D = \{ (\sigma(g)\beta(t), g, t) \mid (g,t)\in G\times K\} =: D_{\sigma, \beta}.
\]
Furthermore we have $L_{D_{\sigma, \beta}} = A$ since for each $t\in K$, the product $\sigma(g)\beta(t)$ runs over all the elements of $G$ as $g$ runs over all the elements of $G$. Hence any covering subgroup $D$ with right invariant $A$ also has left invariant $A$. By Corollary \ref{cor:invariance}, this holds for the opposites also. Hence we conclude that $\widetilde{A} = \{A\}$.

Next we determine the group $\Aut_T(G\times K)$ using Proposition \ref{prop: autta}. First note that restriction from $\Aut_T(G\times K)$ to $\Aut(G)$ is surjective. Indeed, for any $\sigma\in\Aut(G)$, the automorphism $\sigma\times \id$ is in $\Aut_T(G\times K)$. This is also a splitting for the exact sequence of Proposition \ref{prop: autta}. 

On the other hand, the $K$-action on $Z(G)$ (which is the $T_0$-action on $Z(N)$ in the general case) is trivial since for any $t\in K$, we may choose $g_t = 1$. In particular we get $Z^1(K, Z(G)) \cong \Hom(K,Z(G))$. Now it is clear that by Proposition \ref{prop: autta}, we have
\[
\Aut_T(G\times K) \cong \Aut(G)\ltimes \Hom(K, Z(G))
\]
with the action of $\Aut(G)$ on $\Hom(K, Z(G)$ given by post-composition. Finally notice that $\Inn(G\times K) \cong \Inn(G)$ lies in the $\Aut(G)$ factor of $\Aut_T(G\times K)$ and hence we have
\[
\Gamma_{G\times K} = \Out_T(G\times K) \cong \Out(G)\ltimes \Hom(K, Z(G)).
\]
\end{proof}

\end{document}